\documentclass[11pt,a4paper, subeqn, english]{amsart}

\usepackage[english]{babel}
\usepackage[T1]{fontenc}
\usepackage[utf8]{inputenc}
\usepackage{amsmath}
\usepackage{amssymb}
\usepackage{amsthm}
\usepackage{latexsym}
\usepackage{amsfonts}

\usepackage[onehalfspacing]{setspace} 
\setdisplayskipstretch{.421} 
\newlength{\abovebis} 
\newlength{\belowbis} 
\newlength{\aboveshortbis} 
\newlength{\belowshortbis} 
\everydisplay\expandafter{%
  \the\everydisplay 
  \advance\abovedisplayskip\abovebis 
  \advance\belowdisplayskip\belowbis 
  \advance\abovedisplayshortskip\aboveshortbis 
  \advance\belowdisplayshortskip\belowshortbis 
} 

\def\R{\mathbb{R}}

\def\N{\mathbb{N}}
\def\C{\mathbb{C}}

\def\Ree{\mathrm{Re}}
\def\Imm{\mathrm{Im}}

\theoremstyle{plain}

\newtheorem{lem}{Lemma}[section]
\newtheorem{theo}[lem]{Theorem}
\newtheorem{prop}[lem]{Proposition}

\theoremstyle{definition}

\numberwithin{equation}{section}

\begin{document}
\title[Global stability in 2D]{A global stability estimate for the Gel'fand-Calder\'on inverse problem in two dimensions}
\author{Roman G. Novikov}
\author{Matteo Santacesaria}
\address[R. G. Novikov and M. Santacesaria]{Centre de Mathématiques Appliquées, \'Ecole Polytechnique, 91128, Palaiseau, France}
\email{novikov@cmap.polytechnique.fr, santacesaria@cmap.polytechnique.fr}
\begin{abstract}
We prove a global logarithmic stability estimate for the Gel'fand-Calder\'on inverse problem on a two-dimensional domain.
\end{abstract}

\maketitle

\section{Introduction}

Let $D$ be an open bounded domain in $\R^2$ with with $C^2$ boundary and let $v \in C^1(\bar D)$. The Dirichlet-to-Neumann map associated to $v$ is the operator $\Phi : C^1(\partial D) \to L^p(\partial D), \; p < \infty$ defined by:
\begin{equation}
\Phi(f) = \left.\frac{\partial u}{\partial \nu}\right|_{\partial D}
\end{equation}
where $f \in C^1(\partial D)$, $\nu$ is the outer normal of $\partial D$ and $u$ is the $H^1 (\bar D)$-solution of the Dirichlet problem
\begin{equation} \label{equa}
-\Delta u + v(x) u =0 \; \; \textrm{on } D, \; \; u|_{\partial D} = f;
\end{equation}
here we assume that $0$ is not a Dirichlet eigenvalue for the operator $- \Delta + v$ in $D$.

Equation \eqref{equa} arises, in particular, in quantum mechanics, acoustics, electrodynamics; formally, it looks like the Schr\"odinger equation with potential $v$ at zero energy.
\smallskip

The following inverse boundary value problem arises from this construction: given $\Phi$ on $\partial D$, find $v$ on $D$.

This problem can be considered as the Gel’fand inverse boundary value problem for the Schr\"odinger equation at zero energy (see \cite{G}, \cite{N1}) and can also be seen as a generalization of the Calder\'on problem for the electrical impedance tomography (see \cite{C}, \cite{N1}).

The global injectivity of the map $v \to \Phi$ was firstly proved in \cite{N1} for $D \subset \R^d$ with $d \geq 3$ and in \cite{B} for $d = 2$ with $v \in L^p$. A global stability estimate for the Gel'fand-Calder\'on problem for $d \geq 3$ was firstly proved by Alessandrini in \cite{A}; this result was recently improved in \cite{N2}.

In this paper we show that, also in the two dimensional case, an estimate of the same type as in \cite{A} is valid. Indeed out main theorem is the following:

\begin{theo} \label{maintheo}
Let $D \subset \R^2$ be an open bounded domain with $C^2$ boundary, let $v_1 , v_2 \in C^2(\bar D)$ with $\|v_j \|_{C^2(\bar D)} \le N$ for $j =1,2$, and $\Phi_1 , \Phi_2$ the corresponding Dirichlet-to-Neumann operators. For simplicity we assume also that $v_j|_{\partial D} = 0$ and $\frac{\partial}{\partial \nu} v_j|_{\partial D}=0$ for $j=1,2$. Then there exists a constant $C = C(D, N)$ such that
\begin{equation} \label{est}
\|v_2 - v_1\|_{L^{\infty}(D)} \leq C (\log(3 + \|\Phi_2 - \Phi_1\|^{-1} ))^{-\frac 1 2} \log (3 \log(3+\| \Phi_2 -\Phi_1\|^{-1})),
\end{equation}
where $\| A\|$ denotes the norm of an operator $A : L^{\infty}(\partial D) \to L^{\infty}(\partial D)$.
\end{theo}

This is the first result about the global stability of the Gel'fand-Calder\'on inverse problem in two dimension, for general potentials. Results of such a type were only known for special kinds of potentials, e.g. potentials coming from conductivities (see \cite{L} for example). Note also that for the Calder\'on problem (of the electrical impedance tomography) in its initial formulation the global injectivity was firstly proved in \cite{SU} for $d \ge 3$ and in \cite{Na} for $d=2$.
\smallskip

Instability estimates complementing the stability estimates of \cite{A}, \cite{L}, \cite{N2} and of the present work are given in \cite{M}.

The proof of Theorem \ref{maintheo} takes inspiration mostly from \cite{B} and \cite{A}. For $z_0 \in D$ we show existence and uniqueness of a family of solution $\psi_{z_0}(z,\lambda)$ of equation \eqref{equa} where in particular $\psi_{z_0} \to e^{\lambda(z - z_0)^2}$, for $\lambda \to \infty$. This is accomplished by introducing a special Green's function for the Laplacian which satisfies precise estimates. Then, using Alessandrini's identity along with stationary phase techniques, we obtain the result.

\smallskip
An extension of Theorem \ref{maintheo} for the case when we do not assume that $v_j|_{\partial D} = 0$ and $\frac{\partial}{\partial \nu} v_j|_{\partial D}=0$ for $j=1,2$ is given in section \ref{sec6}.

\section{Bukhgeim-type analogues of the Faddeev functions} \label{Bukh}

In this section we introduce the above-mentioned family of solutions of equation \eqref{equa}, which will be used throughout all the paper.
\smallskip

We identify $\R^2$ with $\C$ and use the coordinates $z= x_1 + i x_2, \; \bar z = x_1 - i x_2$ where $(x_1, x_2) \in \R^2$. 
Let us define the function spaces $C^1_{\bar z}(\bar D) = \{ u : u, \frac{\partial u}{ \partial \bar z} \in C(\bar D) \}$ with the norm $\|u \|_{C^1_{\bar z}(\bar D)} = \max ( \|u\|_{C(\bar D)}, \| \frac{\partial u}{\partial \bar z} \|_{C(\bar D)} )$, $C^1_{z}(\bar D) = \{ u : u, \frac{\partial u}{ \partial z} \in C(\bar D) \}$ with an analogous norm and the following functions:
\begin{align} \label{eq11}
G_{z_0}(z,\zeta, \lambda) &= e^{\lambda (z-z_0)^2}g_{z_0} (z,\zeta,\lambda)e^{-\lambda (\zeta-z_0)^2},\\ \label{eq12}
g_{z_0}(z,\zeta, \lambda) &= \frac{e^{\lambda(\zeta-z_0)^2-\bar \lambda(\bar \zeta - \bar z_0)^2}}{4 \pi^2} \int_D \frac{e^{-\lambda(\eta -z_0)^2+\bar \lambda(\bar \eta -\bar z_0)^2}}{(z- \eta)(\bar \eta -\bar \zeta)} d\Ree \eta \, d \Imm \eta, \\ \label{eq13}
\psi_{z_0} (z,\lambda) &= e^{\lambda (z-z_0)^2} \mu_{z_0}(z,\lambda), \\ \label{eq14}
\mu_{z_0} (z,\lambda) &= 1 + \int_D g_{z_0} (z,\zeta, \lambda) v(\zeta) \mu_{z_0} (\zeta, \lambda) d \mathrm{Re}\zeta \, d \mathrm{Im} \zeta ,\\ \label{eq15}
h_{z_0} (\lambda) &= \int_D e^{\lambda (z-z_0)^2 - \bar \lambda (\bar z - \bar z_0)^2} v(z) \mu_{z_0} (z,\lambda) d \mathrm{Re}z \, d \mathrm{Im} z,
\end{align}
where $z , z_0, \zeta \in D$ and $\lambda \in \C$. In addition, equation \eqref{eq14} at fixed $z_0$ and $\lambda$, is considered as a linear integral equation for $\mu_{z_0}( \cdot, \lambda) \in C^1_{\bar z} (\bar D)$.

We have that
\begin{align} \label{eq21}
4 \frac{\partial^2}{\partial z \partial \bar z} &G_{z_0} (z, \zeta, \lambda) = \delta(z-\zeta), \\ \label{eq22}
4\left(\frac{\partial}{\partial z} + 2\lambda (z-z_0) \right) \frac{\partial}{\partial \bar z} &g_{z_0}(z,\zeta, \lambda) = \delta(z-\zeta), \\ \label{eq23}
-4 \frac{\partial^2}{\partial z \partial \bar z} &\psi_{z_0} (z,\lambda) + v(z) \psi_{z_0}(z,\lambda) = 0,\\ \label{eq24}
-4 \left(\frac{\partial}{\partial z} + 2\lambda (z-z_0) \right) \frac{\partial}{\partial \bar z} &\mu_{z_0}(z, \lambda) + v(z) \mu_{z_0} (z,\lambda) = 0,
\end{align}
where $z, z_0 , \zeta \in D$, $\lambda \in \C$, $\delta$ is the Dirac's delta. Formulas \eqref{eq21}-\eqref{eq24} follow from \eqref{eq11}-\eqref{eq14} and from
\begin{equation} \nonumber
\frac{\partial}{\partial \bar z} \frac{1}{\pi z} = \delta(z), \; \;
\left(\frac{\partial}{\partial z} + 2\lambda (z-z_0) \right) \frac{e^{-\lambda(z-z_0)^2+\bar \lambda (\bar z- \bar z_0)^2}}{\pi \bar z} e^{\lambda z_0^2- \bar \lambda \bar z_0^2}  = \delta(z),
\end{equation}
where $z, z_0 , \lambda \in \C$.

We say that the functions $G_{z_0}$, $g_{z_0}$, $\psi_{z_0}$, $\mu_{z_0}$, $h_{z_0}$ are the Bukhgeim-type analogues of the Faddeev functions (see \cite{N1}, \cite{Na}, \cite{B}).

\section{Estimates for $g_{z_0}, \mu_{z_0}, h_{z_0}$}

This section is devoted to crucial estimates concerning the functions defined in section \ref{Bukh}.
\smallskip

Let
\begin{equation} \label{green}
g_{z_0, \lambda} u(z) = \int_D g_{z_0} (z, \zeta, \lambda) u(\zeta) d \mathrm{Re}\zeta \, d \mathrm{Im}\zeta, \; z \in \bar D, \; z_0, \lambda \in \C,
\end{equation}
where $g_{z_0}(z, \zeta, \lambda)$ is defined by \eqref{eq12} and $u$ is a test function.

\begin{lem} \label{lem1}
Let $g_{z_0, \lambda} u$ be defined by \eqref{green}, where $u \in C^1_{\bar z}(\bar D)$, $z_0, \lambda \in \C$. Then the following estimates hold:
\begin{align} \nonumber
&g_{z_0, \lambda} u \in C^1_{\bar z}(\bar D), \\ \label{est1}
\| &g_{z_0, \lambda} u \|_{C^1_{\bar z}(\bar D)} \leq \frac{c_1(D)}{|\lambda|^{\frac 1 2}} \|u \|_{C^1_{\bar z}(\bar D)}, \; |\lambda| \geq 1, \\ \label{est2}
\|&\frac{\partial}{\partial z}g_{z_0,\lambda}u\|_{L^p(\bar D)} \leq \frac{c_2(D,p)}{|\lambda|^{ \frac 1 2}}\|u\|_{C^1_{\bar z} (\bar D)} , \; |\lambda| \geq 1, \; 1 < p < \infty.
\end{align}
\end{lem}
Lemma \ref{lem1} is proved in section \ref{secpf}.

Given a potential $v \in C^1_{\bar z}(\bar D)$ we define the operator $g_{z_0,\lambda} v$ simply as $(g_{z_0,\lambda} v) u(z) = g_{z_0,\lambda} w(z), \; w=vu$, for a test function $u$. If $u \in C^1_{\bar z}(\bar D)$, by Lemma \ref{lem1} we have that $g_{z_0,\lambda} v : C^1_{\bar z}(\bar D)  \to C^1_{\bar z}(\bar D)$,
\begin{equation} \label{estsolmu}
\| g_{z_0,\lambda} v \|^{op}_{C^1_{\bar z}(\bar D)} \leq 2 \| g_{z_0,\lambda} \|^{op}_{C^1_{\bar z}(\bar D)} \|v\|_{C^1_{\bar z}(\bar D)},
\end{equation}
where $\| \cdot \|^{op}_{C^1_{\bar z}(\bar D)}$ denotes the operator norm in $C^1_{\bar z}(\bar D)$, $z_0, \lambda \in \C$. In addition, $\| g_{z_0,\lambda} \|^{op}_{C^1_{\bar z}(\bar D)}$ is estimated in Lemma \ref{lem1}. Inequality \eqref{estsolmu} and Lemma \ref{lem1} implies existence and uniqueness of $\mu_{z_0}(z, \lambda)$ (and thus also $\psi_{z_0}(z,\lambda)$) for $|\lambda|$ sufficiently large.
\smallskip

Let
\begin{align*}
\mu^{(k)}_{z_0}(z, \lambda) &= \sum_{j=0}^k (g_{z_0,\lambda} v)^j 1, \\
h^{(k)}_{z_0}(\lambda) &= \int_D e^{\lambda (z-z_0)^2 -\bar \lambda (\bar z- \bar z_0)^2}v(z) \mu^{(k)}_{z_0}(z,\lambda) d\Ree z \, d \Imm z,
\end{align*} 
where $z,z_0 \in D$, $\lambda \in \C$, $k \in \N \cup \{ 0 \}$.

\begin{lem} \label{lem2}
For $v \in C^1_{\bar z}(\bar D)$ such that $v|_{\partial D} = 0$ the following formula holds:
\begin{equation} \label{estp}
v(z_0) = \frac{2}{\pi} \lim_{\lambda \to \infty} |\lambda| h^{(0)}_{z_0}(\lambda), \; z_0 \in D.
\end{equation}
In addition, if $v \in C^2(\bar D)$, $v|_{\partial D}= 0$ and $\frac{\partial v}{\partial \nu}|_{\partial D} = 0$ then
\begin{equation} \label{estm}
|v(z_0) - \frac{2}{\pi}|\lambda| h^{(0)}_{z_0}(\lambda)| \leq c_3(D) \frac{\log(3|\lambda|)}{|\lambda|}\|v\|_{C^2(\bar D)},
\end{equation}
for $z_0 \in D$, $\lambda \in \C$, $|\lambda| \geq 1$.
\end{lem}
Lemma \ref{lem2} is proved in section \ref{secpf}.

Let
$$W_{z_0}(\lambda)=\int_D e^{\lambda(z-z_0)^2-\bar\lambda(\bar z-\bar z_0)^2}
w(z)d\Ree\,z d\Imm\,z,$$
where $z_0\in\bar D$, $\lambda\in\C$ and $w$ is some function on $\bar D$.
(One can see that $W_{z_0}=h_{z_0}^{(0)}$ for $w=v$.)

\begin{lem} \label{lem3}
For $w\in C_{\bar z}^1(\bar D)$ the following estimate holds:
\begin{subequations} \label{estw}
\begin{equation} 
|W_{z_0}(\lambda)|\le c_4(D) \frac{\log\,(3|\lambda|)}{ |\lambda|}  \|w\|_{C_{\bar z}^1(\bar D)},\ z_0\in\bar D,\ |\lambda|\ge 1,
\end{equation}
\begin{equation}
|W_{z_0}(\lambda)|\le c_{4,1}(D) \frac{\log\,(3|\lambda|)}{ |\lambda|}  \|w\|_{C(\bar D)}+\frac{c_{4,2}(D,p)}{|\lambda|}\| \frac{\partial}{\partial z} w\|_{L^p(\bar D)},
\end{equation}
for  $2 < p < \infty$.
\end{subequations}
\end{lem}
Lemma \ref{lem3} is proved in Section 5.

\begin{lem} \label{lem4}
For $v \in C^1_{\bar z}(\bar D)$ and for $\|g_{z_0,\lambda} v \|^{op}_{C^1_{\bar z}(\bar D)} \leq \delta < 1$ we have that
\begin{align} \label{estmuk}
&\|\mu_{z_0}(\cdot, \lambda) - \mu_{z_0}^{(k)}(\cdot, \lambda)\|_{C^1_{\bar z}(\bar D)} \leq \frac{\delta^{k+1}}{1-\delta}, \\ \label{esth}
&|h_{z_0}(\lambda)-h^{(k)}_{z_0}(\lambda)| \leq c_4(D)\frac{\log(3|\lambda|)}{|\lambda|} \frac{\delta^{k+1}}{1-\delta} \|v\|_{C^1_{\bar z}(\bar D)},
\end{align}
where $z_0 \in D \setminus \{ 0 \}, \; \lambda \in \C,\; |\lambda |\ge 1, \; k \in \N \cup \{ 0 \}$.
\end{lem}
Lemma \ref{lem4} is proved in section \ref{secpf}.

\section{Proof of Theorem \ref{maintheo}}

We start from Alessandrini's identity
\begin{align} \nonumber
\int_{D} (v_2 (z) - v_1(z)) \psi_2(z) &\psi_1 (z) d\Ree z \, d\Imm z \\ \nonumber 
&= \int_{\partial D} \int_{\partial D} \psi_1(z) (\Phi_2 - \Phi_1)(z, \zeta) \psi_2(\zeta) |d\zeta| |dz|,
\end{align}
which holds for every $\psi_j$ solution of $(-\Delta + v_j)\psi_j = 0$ on $D$, $j =1,2$. Here $(\Phi_2 - \Phi_1)(z, \zeta)$ is the kernel of the operator $\Phi_2 - \Phi_1$.

Let $\bar\mu_{z_0}$ denote the complex conjugated of $\mu_{z_0}$ for real-valued $v$ and, more generally, the solution of \eqref{eq14} with $g_{z_0}(z,\zeta,\lambda)$ replaced by $\overline{g_{z_0}(z,\zeta,\lambda)} $for complex-valued $v$.
Put $\psi_1(z) = \bar \psi_{1,z_0} (z, -\lambda) = e^{-\bar \lambda (\bar z - \bar z_0)^2} \bar \mu_1(z, -\lambda)$, $\psi_2(z) = \psi_{2,z_0} (z, \lambda) = e^{\lambda (z -  z_0)^2}  \mu_2(z, \lambda)$, where we called for simplicity $\bar \mu_1 = \bar \mu_{1,z_0}, \; \mu_2 = \mu_{2,z_0}$. This gives
\begin{align} \label{aless1}
\int_{D} &e_{\lambda, z_0} (z) (v_2 (z) - v_1(z)) \mu_2(z,\lambda) \bar \mu_1 (z,\lambda) d\Ree z \, d\Imm z \\ \nonumber 
= \int_{\partial D} &\int_{\partial D} e^{-\bar \lambda (\bar z - \bar z_0)^2} \bar \mu_1(z, -\lambda) (\Phi_2 - \Phi_1)(z, \zeta)  e^{\lambda (\zeta -  z_0)^2}  \mu_2(\zeta, \lambda) |d\zeta| |dz|,
\end{align}
where $e_{\lambda, z_0} (z) = e^{\lambda (z-z_0)^2 - \bar \lambda(\bar z- \bar z_0)^2}$. The left side $I (\lambda)$ of \eqref{aless1} can be written as the sum of four integrals, namely
\begin{align*}
I_1 (\lambda) &= \int_D e_{\lambda, z_0} (z) (v_2(z) - v_1(z)) d \Ree z \, d \Imm z, \\
I_2 (\lambda) &= -\int_D e_{\lambda, z_0} (z) (v_2(z) - v_1(z))(\mu_2 - 1) ( \bar \mu_1 -1) d \Ree z \, d \Imm z, \\
I_3 (\lambda) &= -I_2(\lambda) + \int_D e_{\lambda, z_0} (z) (v_2(z) - v_1(z))( \mu_2 -1) d \Ree z \, d \Imm z, \\
I_4 (\lambda) &= -I_2(\lambda) + \int_D e_{\lambda, z_0} (z) (v_2(z) - v_1(z))(\bar \mu_1 -1) d \Ree z \, d \Imm z,
\end{align*}
for $z_0 \in D$. By Lemma \ref{lem1}, \ref{lem2}, \ref{lem3}, \ref{lem4} we have the following estimates:
\begin{align} \label{est91}
&\left| \frac{2}{\pi}|\lambda| I_1 - (v_2(z_0) - v_1(z_0)) \right| \leq c_3( D) \frac{\log (3 |\lambda|)}{|\lambda|} \|v_2 - v_1\|_{C^2(\bar D)}, \\ \label{est92}
&|I_2| \leq c_5(D) \frac{\log(3 |\lambda|)}{|\lambda|^{\frac 3 2}} \|v_2 - v_1\|_{C^1(\bar D)} \|v_1\|_{C^1_{ z} (\bar D)} \|v_2\|_{C^1_{\bar z} (\bar D)}, \\ \label{est93}
&|I_3| \leq |I_2| + c_6(D)  \frac{\log(3 |\lambda|)}{|\lambda|^{\frac 3 2}} \|v_2 - v_1\|_{C^1_{\bar z} (\bar D)} \|v_2\|_{C^1_{\bar z} (\bar D)}, \\ \label{est94}
&|I_4| \leq |I_2| + c_6(D)  \frac{\log(3 |\lambda|)}{|\lambda|^{\frac 3 2}} \|v_2 - v_1\|^2_{C^1_{z} (\bar D)} \|v_1\|_{C^1_{z} (\bar D)},
\end{align}
for $|\lambda|$ sufficiently large for example, for $\lambda$ such that
\begin{align} \label{estlambda}
\frac{2\,c_1(D)}{ |\lambda|^{\frac 1 2}} \max\,(\|v_1\|_{C_{\bar z}^1(\bar D)},\|v_1\|_{C_{z}^1(\bar D)},\|v_2\|_{C_{\bar z}^1(\bar D)}, \|v_2\|_{C_{z}^1(\bar D)})\le \frac 1 2,\ |\lambda|\ge 1.
\end{align}
The right side $J(\lambda)$ of \eqref{aless1} can be estimated as follows:
\begin{align} \label{est95}
|\lambda||J(\lambda)| \leq c_7(D) e^{(2  L^2+1) |\lambda|} \|\Phi_2 - \Phi_1\|,
\end{align}
where we called $L = \max_{z \in \partial D, \; z_0 \in D} |z-z_0|$.

Putting together estimates \eqref{est91}-\eqref{est95} we obtain
\begin{align} \label{est96}
|v_2(z_0) - v_1(z_0)| &\leq c_8(D) \frac{\log(3 |\lambda|)}{|\lambda|^{\frac 1 2}}N^3 + \frac{2}{\pi}c_7(D) e^{(2  L^2+1) |\lambda|} \|\Phi_2 - \Phi_1\|
\end{align}
for $z_0 \in D$ and $N$ is the costant in the statement of Theorem \ref{maintheo}. We call $\varepsilon =  \|\Phi_2 - \Phi_1\|$ and impose $|\lambda| = \gamma \log (3+\varepsilon^{-1})$, where $0 < \gamma < (2L^2+1)^{-1}$ so that \eqref{est96} reads
\begin{align}
|v_2(z_0) - v_1(z_0)| &\leq c_8(D)N^3 (\gamma \log(3+\varepsilon^{-1}))^{- \frac 1 2} \log(3 \gamma \log(3+\varepsilon^{-1}))  \\ \nonumber
&+ \frac{2}{\pi}c_7(D) (3+ \varepsilon^{-1})^{(2  L^2+1)\gamma }\varepsilon,
\end{align}
for every $z_0 \in D$, with
\begin{equation} \label{esteps}
0 < \varepsilon \leq \varepsilon_1 (D, N, \gamma),
\end{equation}
where $\varepsilon_1$ is sufficiently small or, more precisely, where \eqref{esteps} implies that $|\lambda| =\gamma \log (3+ \varepsilon^{-1})$ satisfies \eqref{estlambda}.

As $(3+ \varepsilon^{-1})^{(2  L^2+1) \gamma} \varepsilon \to 0$ for $\varepsilon \to 0$ more rapidly then the other term, we obtain that
\begin{align} \label{estest}
\|v_2 -v_1\|_{L^{\infty}(D)} \leq c_{9}(D,N,\gamma) \frac{\log (3\log (3+\|\Phi_2 - \Phi_1\|^{-1}))}{(\log (3+ \|\Phi_2 - \Phi_1\|^{-1}))^{\frac 1 2 }} 
\end{align}
for $\varepsilon = \|\Phi_2 - \Phi_1\| \leq \varepsilon_1(D, N,\gamma)$.

Estimate \eqref{estest} for general $\varepsilon$ (with modified $c_{10}$) follows from \eqref{estest} for $\varepsilon \leq \varepsilon_1(D, N, \gamma)$ and the assumption that $\|v_j\|_{L^{\infty}(D)} \leq N, \; j = 1,2$. This completes the proof of Theorem \ref{maintheo}.

\section{Proofs of the Lemmata} \label{secpf}

\begin{proof}[Proof of Lemma \ref{lem1}]
One can see that $g_{z_0, \lambda} = \frac 1 4 T \bar T_{z_0, \lambda}$, for $z_0, \lambda \in \C$, where
\begin{align}
&T u(z) = - \frac{1}{\pi} \int_D \frac{u(\zeta)}{\zeta - z} d \Ree\zeta \, d \Imm \zeta, \\
&\bar T_{z_0, \lambda} u(z) = -\frac{e^{-\lambda(z- z_0)^2 + \bar \lambda(\bar z - \bar z_0)^2}}{\pi} \int_D \frac{e^{\lambda(\zeta-z_0)^2-\bar \lambda(\bar \zeta - \bar z_0)^2}}{\bar \zeta- \bar z} u(\zeta)  d \Ree\zeta \, d \Imm \zeta,
\end{align}
for $z \in \bar D$ and $u$ a test function. Estimates \eqref{est1}, \eqref{est2} now follow from
\begin{equation} \label{est31}
Tw \in C^1_{\bar z}(\bar D),
\end{equation}
\begin{subequations} \label{est32}
\begin{equation}
\| Tw \|_{C^1_{\bar z}(\bar D)} \leq n_1(D) \| w \|_{C(\bar D)}, \; \textrm{where } w \in C(D),
\end{equation}
\begin{equation}
\|\frac{\partial T}{\partial z} w\|_{L^p(\bar D)} \leq n(D,p) \|w\|_{L^p(\bar D)}, \; 1 < p < \infty,
\end{equation}
\end{subequations}
\begin{align} \label{est33}
&\bar T_{z_0,\lambda} u \in C(\bar D), \\ \label{est34}
&\|\bar T_{z_0, \lambda} u\|_{C(\bar D)} \leq \frac{n_2(D)}{|\lambda|^{\frac 1 2}} \|u \|_{C^1_{\bar z}(\bar D)}, \; |\lambda| \geq 1, \\ \label{est35}
&\|\bar T_{z_0, \lambda} u\|_{C(\bar D)} \leq \frac{\log (3 |\lambda|)(1+|z-z_0|)n_3(D)}{|\lambda| |z- z_0|^2}\|u\|_{C^1_{\bar z}(\bar D)}, \; |\lambda| \geq 1,
\end{align}
where $u \in C^1_{\bar z}(\bar D)$, $z_0, \lambda \in \C$. Estimates \eqref{est31}, \eqref{est32} are well-known (see \cite{Vekua}).

The assumption $u \in C^1_{\bar z}(\bar D)$ is not necessary at all for \eqref{est33}: indeed, using well-known arguments it is sufficient to take $u \in C(\bar D)$.

Let us prove \eqref{est34} and \eqref{est35}. We have that
$$- \pi e^{\lambda(z- z_0)^2 - \bar \lambda(\bar z - \bar z_0)^2} \bar T_{z_0,\lambda} u(z) = I_{z_0, \lambda, \varepsilon}(z) + J_{z_0,\lambda, \varepsilon}(z),$$
where
\begin{align}
I_{z_0, \lambda, \varepsilon} (z) &= \int_{D \cap (B_{z,\varepsilon} \cup B_{z_0, \varepsilon})} \! \! \! \! \! \! \! \! \! \! \! \!  \frac{e^{\lambda(\zeta -z_0)^2-\bar \lambda(\bar \zeta - \bar z_0)^2}}{\bar \zeta - \bar z} u(\zeta)  d \Ree\zeta \, d \Imm \zeta, \\
J_{z_0,\lambda, \varepsilon} (z) &= \int_{D_{z,z_0,\varepsilon}} \frac{e^{\lambda(\zeta -z_0)^2-\bar \lambda(\bar \zeta - \bar z_0)^2}}{\bar \zeta - \bar z} u(\zeta)  d \Ree\zeta \, d \Imm \zeta,
\end{align}
and $B_{z,\varepsilon} = \{ \zeta \in \C : |\zeta -z| < \varepsilon \}$, $D_{z,z_0,\varepsilon} = D \setminus (B_{z,\varepsilon} \cup B_{z_0, \varepsilon})$. One sees that
\begin{equation} \label{estI}
|I_{z_0,\lambda, \varepsilon}(z) | \leq 2 \int_{B_{z,\varepsilon}} \frac{\|u\|_{C(\bar D)}}{|\zeta - z|} d \Ree \zeta \, d \Imm \zeta = 4\pi \varepsilon \|u\|_{C(\bar D)}, 
\end{equation}
with $z,z_0,\lambda \in \C$, $\varepsilon >0$.
Further, we have that
\begin{align*} 
J_{z_0,\lambda,\varepsilon} (z) &= -\frac{1}{2 \bar \lambda} \int_{D_{z,z_0,\varepsilon}} \frac{\partial}{\partial \bar \zeta}\left( e^{\lambda(\zeta -z_0)^2-\bar \lambda(\bar \zeta - \bar z_0)^2} \right) \frac{u(\zeta)}{(\bar \zeta - \bar z)(\bar \zeta - \bar z_0)} d\Ree \zeta \, d\Imm \zeta \\
&= J^1_{z_0, \lambda, \varepsilon}(z) + J^2_{z_0, \lambda, \varepsilon}(z),
\end{align*}
where
\begin{align*}
&J^1_{z_0, \lambda, \varepsilon}(z) = -\frac{1}{4 i \bar \lambda} \int_{\partial D_{z,z_0,\varepsilon}} \frac{e^{\lambda(\zeta -z_0)^2-\bar \lambda(\bar \zeta - \bar z_0)^2} }{(\bar \zeta - \bar z)(\bar \zeta -\bar z_0)}  u(\zeta)d\zeta, \\
&J^2_{z_0, \lambda, \varepsilon}(z) =  \frac{1}{2 \bar \lambda} \int_{D_{z,z_0,\varepsilon}} \! \! \! \! \! \! \! \! \! e^{\lambda(\zeta -z_0)^2-\bar \lambda(\bar \zeta - \bar z_0)^2} \frac{\partial}{\partial \bar \zeta}\left( \frac{u(\zeta)}{(\bar \zeta - \bar z)(\bar \zeta - \bar z_0)}\right)  d\Ree \zeta \, d\Imm \zeta 
\end{align*}
Now we get
\begin{align} \nonumber
&|J^1_{z_0,\lambda, \varepsilon} (z)| \leq M^1_{z,z_0, \lambda, \varepsilon} := \frac{1}{4 |\lambda|} \int_{\partial D_{z,z_0,\varepsilon}} \frac{| u(\zeta) | |d \zeta|}{|\bar \zeta - \bar z||\bar \zeta - \bar z_0|}, \\ \label{est41}
&M^1_{z,z_0,\lambda, \varepsilon} \leq \frac{1}{8 |\lambda|}\int_{\partial D_{z,z_0,\varepsilon}} \left( \frac{1}{|\bar \zeta - \bar z|^2} + \frac{1}{|\bar \zeta - \bar z_0|^2} \right) |d\zeta| \|u\|_{C(D)}, \\ \label{est42}
&M^1_{z,z_0,\lambda, \varepsilon} \leq \frac{1}{2 |z-z_0| |\lambda|}\int_{\partial D_{z,z_0,\varepsilon}} \left( \frac{1}{|\bar \zeta - \bar z|} + \frac{1}{|\bar \zeta - \bar z_0|} \right) |d\zeta| \|u\|_{C(D)}.
\end{align}
We also have
\begin{align}\nonumber
|J^2_{z_0,\lambda, \varepsilon} (z)| \leq M^2_{z,z_0, \lambda, \varepsilon} :&= \frac{1}{2 |\lambda|} \int_{D_{z,z_0,\varepsilon}} \frac{| \frac{\partial u}{\partial \bar \zeta}(\zeta) |}{|\bar \zeta - \bar z||\bar \zeta -\bar z_0|} + \frac{| u(\zeta) |}{|\bar \zeta - \bar z|^2|\bar \zeta -\bar z_0|} \\ \nonumber
&\qquad + \frac{| u(\zeta) |}{|\bar \zeta - \bar z||\bar \zeta -\bar z_0|^2} d\Ree \zeta \, d\Imm \zeta,
\end{align}
\begin{align} \label{est43}
M^2_{z,z_0, \lambda, \varepsilon} &\leq \frac{1}{2 |\lambda|} \int_{D_{z,z_0,\varepsilon}} \frac{| \frac{\partial u}{\partial \bar \zeta}(\zeta) |}{|\bar \zeta - \bar z|^2} + \frac{| \frac{\partial u}{\partial \bar \zeta}(\zeta) |}{|\bar \zeta - \bar z_0|^2} + 2 \frac{| u(\zeta) |}{|\bar \zeta - \bar z|^3} \\ \nonumber &\qquad + 2\frac{| u(\zeta) |}{|\bar \zeta - \bar z_0|^3} d\Ree \zeta \, d\Imm \zeta
\end{align}
\begin{align} \label{est44}
M^2_{z,z_0, \lambda, \varepsilon} \leq \frac{1}{2 |\lambda|} \int_{D_{z,z_0,\varepsilon}} \frac{2|\frac{\partial u}{\partial \bar \zeta}(\zeta) |}{|\bar \zeta - \bar z||z-z_0|} + \frac{2| \frac{\partial u}{\partial \bar \zeta}(\zeta) |}{|\bar \zeta -\bar z_0||z-z_0|} +  \frac{2| u(\zeta) |}{|\bar \zeta - \bar z|^2 |z-z_0|} \\ \nonumber +  \frac{4| u(\zeta) |}{|\bar \zeta - \bar z||z- z_0|^2} +  \frac{2| u(\zeta) |}{|\bar \zeta - \bar z_0|^2|z-z_0|} +  \frac{4| u(\zeta) |}{|\bar \zeta - \bar z_0||z-z_0|^2} d\Ree \zeta \, d\Imm \zeta.
\end{align}
Using \eqref{est41} and \eqref{est43} we obtain that
\begin{align} \label{est51}
&|J^1_{z_0,\lambda, \varepsilon}(z)| \leq |\lambda|^{-1} n_4(D) \varepsilon^{-1} \|u\|_{C(D)}, \\ \label{est52}
&|J^2_{z_0,\lambda, \varepsilon}(z)| \leq |\lambda|^{-1} n_5(D) \varepsilon^{-1} \|u\|_{C(D)} + |\lambda|^{-1} n_6(D)\log(3 \varepsilon^{-1}) \|\frac{\partial u}{\partial \bar z}\|_{C(D)},
\end{align}
where $z,z_0,\lambda \in \C$, $|\lambda| \geq 1$, $0 < \varepsilon < 1$.

If $z_0 \neq z$ we can use \eqref{est42} and \eqref{est44} in order to obtain
\begin{align} \label{est53}
|J^1_{z_0,\lambda, \varepsilon}(z)| &\leq |\lambda|^{-1} |z-z_0|^{-1} n_7(D) \log(3 \varepsilon^{-1}) \|u\|_{C(D)}, \\ \label{est54}
|J^2_{z_0,\lambda, \varepsilon}(z)| &\leq |\lambda|^{-1} |z-z_0|^{-2} n_8(D) \log(3 \varepsilon^{-1}) \|u\|_{C(D)} \\ \nonumber
&+ |\lambda|^{-1} |z-z_0|^{-1}n_9(D) \|\frac{\partial u}{\partial \bar z}\|_{C(D)},
\end{align}

Finally, putting $\varepsilon = |\lambda|^{-\frac 1 2}$ into \eqref{estI}, \eqref{est51}, \eqref{est52} we obtain \eqref{est34}, while putting $\varepsilon = |\lambda|^{-1}$ into \eqref{estI}, \eqref{est53}, \eqref{est54} we obtain \eqref{est35}. The proof follows.
\end{proof}

\begin{proof}[Proof of Lemma \ref{lem2}]
First we extend our potential $v$ to a larger domain $D_1 \supset D$ (always with $C^2$ boundary) such that $dist(\partial D_1, \partial D) \ge \delta > 0 $ (for some $\delta$) by putting $v|_{D_1 \setminus D} \equiv 0$. In such a way $v \in C^1(D_1) \cap C^2(D_1 \setminus \partial D)$ with $\|v\|_{C^{k}(D_1)} = \|v\|_{C^{k}(D)}$ for $k=1,2$.

Now let $\chi_{\delta}$ be a real-valued function on $\C$, with $\delta > 0$, constructed as follows:
$$
\begin{array}{l}
\chi_{\delta} (z) = \chi(z / \delta), \textrm{ where} \\
\chi \in C^{\infty}(\C), \; \chi \textrm{ is real valued}, \\
\chi(z) = \chi(|z|), \\
\chi(z) \equiv 1 \textrm{ for } |z| \leq 1/2, \\
\chi(z) \equiv 0 \textrm{ for } |z| \geq 1.
\end{array}
$$

Let 
$$v_{lin} (z,z_0) = v(z_0) + v_z (z_0)(z-z_0) + v_{\bar z}(z_0) (\bar z - \bar z_0),$$
for $z, z_0 \in D_1$, $v_z = \frac{\partial v}{\partial z}$ and $v_{\bar z} = \frac{\partial v}{\partial \bar z}$.

We can write $h^{(0)}_{z_0}(\lambda) = S_{z_0, \delta} (\lambda) + R_{z_0,\delta} (\lambda)$, where
\begin{align*}
S_{z_0,\delta} (\lambda) &= \int_{\C} e^{\lambda (z-z_0)^2 - \bar \lambda (\bar z- \bar z_0)^2} v_{lin} (z,z_0) \chi_{\delta} (z-z_0) d\Ree z \, d \Imm z \\
&= \int_{\C} e^{i |\lambda|(z^2 + \bar z^2)} v_{lin} (e^{-i\varphi (\lambda)}z +z_0, z_0) \chi_{\delta} (z) d\Ree z \, d \Imm z, \\
R_{z_0, \delta}(\lambda) &= \int_{D_1} e^{\lambda (z-z_0)^2 - \bar \lambda (\bar z- \bar z_0)^2} \left(v(z) -v_{lin} (z,z_0) \chi_{\delta} (z-z_0)\right) d\Ree z \, d \Imm z
\end{align*}
where $\varphi(\lambda) = \frac 1 2 (\mathrm{arg}(\lambda) - \frac{\pi}{2})$, $z_0 \in D$, $\lambda \in \C$.

Using the stationary phase method we obtain that
\begin{align} \label{est65}
&v(z_0) = \frac{2}{\pi} \lim_{\lambda \to \infty} |\lambda| S_{z_0, \delta} (\lambda), \\ \label{est66}
&|v(z_0) - \frac{2}{\pi} |\lambda| S_{z_0,\delta} (\lambda)| \leq q_1(D, \delta) \|v\|_{C^1(\bar D)} |\lambda|^{-1},
\end{align}
$z_0 \in D$, $\delta >0$, $\lambda \in \C$, $|\lambda| \geq 1$.
Integrating by parts we can write
\begin{align*}
R_{z_0, \delta}(\lambda) &= -\frac{1}{2 \bar \lambda} \int_{D_1} \frac{\partial}{\partial \bar z}\left( e^{\lambda (z-z_0)^2 - \bar \lambda (\bar z- \bar z_0)^2} \right) \\
&\times \frac{\left(v(z) -v_{lin} (z,z_0) \chi_{\delta} (z-z_0)\right)}{\bar z - \bar z_0} d\Ree z \, d \Imm z = R^1_{z_0, \delta}(\lambda) + R^2_{z_0, \delta}(\lambda), \\
R^1_{z_0, \delta}(\lambda) &= \frac{-1}{4i \bar \lambda} \int_{\partial {D_1}} \! \! \! \! e^{\lambda (z-z_0)^2 - \bar \lambda (\bar z- \bar z_0)^2}  \frac{\left(v(z) -v_{lin} (z,z_0) \chi_{\delta} (z-z_0)\right)}{\bar z - \bar z_0} dz, \\
R^2_{z_0, \delta}(\lambda) &= \frac{1}{2 \bar \lambda} \int_{D_1} e^{\lambda (z-z_0)^2 - \bar \lambda (\bar z- \bar z_0)^2} \\
&\times \frac{\partial}{\partial \bar z} \left(\frac{\left(v(z) -v_{lin} (z,z_0) \chi_{\delta} (z-z_0)\right)}{\bar z - \bar z_0} \right) d\Ree z \, d \Imm z,
\end{align*}
for $z_0 \in D$, $\lambda \in \C \setminus \{ 0 \}$. In addition, we have that
\begin{align} \label{est61}
\lim_{\lambda \to \infty} |\lambda| R^1_{z_0,\delta} (\lambda) =0, \\ \label{est62}
\lim_{\lambda \to \infty} |\lambda| R^2_{z_0,\delta} (\lambda) =0.
\end{align}

Formula \eqref{est61} follows from properties of $\chi_{\delta}$, the assumption that $z_0 \in D$ and that $v|_{\partial D_1} \equiv 0$. Actually, as a corollary of this properties we have that $v(z) - v_{lin}(z,z_0)\chi_{\delta}(z-z_0) \equiv 0$ for $z \in \partial D_1$ and, therefore, $R^1_{z_0,\delta} (\lambda) \equiv 0$ for $\lambda \in \C \setminus \{ 0 \}$.

Formula \eqref{est62} for $v \in C^1(\bar D_1)$ is a consequence of the estimates
\begin{align} \label{est63}
&R^{2,1}_{z_0, \delta, \varepsilon}(\lambda) := \int_{B_{z_0,\varepsilon}} e^{\lambda (z-z_0)^2 - \bar \lambda (\bar z- \bar z_0)^2} \\ \nonumber
&\times \frac{\partial}{\partial \bar z} \left(\frac{\left(v(z) -v_{lin} (z,z_0) \chi_{\delta} (z-z_0)\right)}{\bar z - \bar z_0} \right) d\Ree z \, d \Imm z = O (\varepsilon) \; \textrm{as } \varepsilon \to 0 \\ \label{est64}
&R^{2,2}_{z_0, \delta, \varepsilon}(\lambda) := \int_{D_{z_0,\varepsilon}} e^{\lambda (z-z_0)^2 - \bar \lambda (\bar z- \bar z_0)^2} \\ \nonumber
&\times \frac{\partial}{\partial \bar z} \left(\frac{\left(v(z) -v_{lin} (z,z_0) \chi_{\delta} (z-z_0)\right)}{\bar z - \bar z_0} \right) d\Ree z \, d \Imm z \to 0 \; \textrm{as } \lambda \to \infty
\end{align}
where $B_{z_0, \varepsilon} = \{ z \in \C : |z-z_0| < \varepsilon \}$, $D_{z_0,\varepsilon} = D_1 \setminus B_{z_0, \varepsilon}$. In \eqref{est63}-\eqref{est64} we assume that $z_0 \in D, \; 0 < \varepsilon < \delta, \; \lambda \in \C$.

Estimate \eqref{est63} is obtained by standard arguments using that
$$|v(z) - v(z_0) | \leq \|v\|_{C^1(\bar D)} |z-z_0|, \; z_0 \in D, \; z \in B_{z_0,\delta},$$ 
while \eqref{est64} is a variation of the Riemann-Lebesgue Lemma.

Formula \eqref{estp} now follows from \eqref{est65}, \eqref{est61}, \eqref{est62}.

Under the assumptions mentioned in Lemma \ref{lem2}, the final part of the proof of estimate \eqref{estm} consists in the following. We have, for $\varepsilon < \delta /2$,
\begin{align} \label{est67}
|R^{2,1}_{z_0, \delta, \varepsilon}(\lambda)| &\leq \int_{B_{z_0, \varepsilon}} \frac{|v(z) - v_{lin} (z,z_0)|}{|z-z_0|^2} d\Ree z \, d \Imm z \\ \nonumber
&+ \int_{B_{z_0, \varepsilon}} \frac{|v_{\bar z}(z) - v_{\bar z} (z_0)|}{|z-z_0|} d\Ree z \, d \Imm z \leq \frac 7 2 \pi \|v\|_{C^2(\bar D)} \varepsilon^2, \\ \nonumber
R^{2,2}_{z_0, \delta, \varepsilon}(\lambda) &= \frac{-1}{2 \bar \lambda} \int_{D_{z_0,\varepsilon}} \frac{\partial}{\partial \bar z} \left( e^{\lambda (z-z_0)^2 - \bar \lambda (\bar z- \bar z_0)^2} \right) \frac{1}{\bar z - \bar z_0} \\ \nonumber
&\times \frac{\partial}{\partial \bar z} \left(\frac{\left(v(z) -v_{lin} (z,z_0) \chi_{\delta} (z-z_0)\right)}{\bar z - \bar z_0} \right) d\Ree z \, d \Imm z \\ \nonumber
&= \frac{-1}{2 \bar \lambda}(R^{2,2,1}_{z_0, \delta, \varepsilon}(\lambda)+ R^{2,2,2}_{z_0, \delta, \varepsilon}(\lambda)), \\ \nonumber
R^{2,2,1}_{z_0, \delta, \varepsilon}(\lambda) &= \frac{1}{2 i} \int_{\partial D_{z_0,\varepsilon}} e^{\lambda (z-z_0)^2 - \bar \lambda (\bar z- \bar z_0)^2} \frac{1}{\bar z - \bar z_0} \\ \nonumber
&\times \frac{\partial}{\partial \bar z} \left(\frac{\left(v(z) -v_{lin} (z,z_0) \chi_{\delta} (z-z_0)\right)}{\bar z - \bar z_0} \right) dz \\ \nonumber
= \frac{-1}{2 i} \int_{\partial B_{z_0,\varepsilon}} &e^{\lambda (z-z_0)^2 - \bar \lambda (\bar z- \bar z_0)^2} \frac{1}{\bar z - \bar z_0} \frac{\partial}{\partial \bar z} \left(\frac{v(z) -v_{lin} (z,z_0)}{\bar z - \bar z_0} \right) dz,
\end{align}
where we used in particular that $v|_{\partial D_1} \equiv 0, \frac{\partial}{\partial \nu} v|_{\partial D_1} \equiv 0$,
\begin{align*}
R^{2,2,2}_{z_0, \delta, \varepsilon}(\lambda) &= -\int_{D_{z_0,\varepsilon}}  e^{\lambda (z-z_0)^2 - \bar \lambda (\bar z- \bar z_0)^2} \\
&\times \frac{\partial}{\partial \bar z} \left( \frac{1}{\bar z - \bar z_0} \frac{\partial}{\partial \bar z} \left(\frac{\left(v(z) -v_{lin} (z,z_0) \chi_{\delta} (z-z_0)\right)}{\bar z - \bar z_0} \right)  \right) d\Ree z \, d \Imm z.
\end{align*}
We have, for $\varepsilon < \delta /2$
\begin{align} \label{est68}
|R^{2,2,1}_{z_0, \delta, \varepsilon}(\lambda)| &\leq \frac{1}{2} \int_{\partial B_{z_0,\varepsilon}} \frac{|v(z)-v_{lin}(z,z_0)|}{|z-z_0|^3}|dz| \\ \nonumber
&+ \frac{1}{2} \int_{\partial B_{z_0,\varepsilon}} \frac{|v_{\bar z}(z)-v_{\bar z}(z_0)|}{|z-z_0|^2}|dz| \leq \frac 7 2 \pi \|v\|_{C^2(\bar D)}, \\
|R^{2,2,2}_{z_0, \delta, \varepsilon}(\lambda)| &\leq |R^{2,2,2}_{z_0, \delta, \delta/2}(\lambda)| + |R^{2,2,2}_{z_0, \delta, \varepsilon}(\lambda) - R^{2,2,2}_{z_0, \delta, \delta/2}(\lambda)|, \\
|R^{2,2,2}_{z_0, \delta, \delta/2}(\lambda)| &\leq q_2(D,\delta) \|v\|_{C^2(\bar D)},\\ \nonumber
|R^{2,2,2}_{z_0, \delta, \varepsilon}(\lambda) - &R^{2,2,2}_{z_0, \delta, \delta/2}(\lambda)| \leq \sum_{j=1}^5 \int_{B_{z_0, \delta/2} \setminus B_{z_0, \varepsilon}} \! \! \! \! \! \! \! \! \! \! \! \! \! \! \! \! \! \! \! \! \! \! \! \! \! u_j (z,z_0) d \Ree z \, d \Imm z,
\end{align}
with
\begin{align} \label{est80}
u_1(z,z_0) &= \frac{1}{|z-z_0|^2} \left| \frac{v_{\bar z}(z) - v_{\bar z}(z_0)}{\bar z - \bar z_0} \right|, \\ \label{est81}
u_2(z,z_0) &= \frac{1}{|z-z_0|^2} \left| \frac{v(z) - v_{lin}(z,z_0)}{(\bar z - \bar z_0)^2} \right|,  \\ \label{est82}
u_3(z,z_0) &= \frac{1}{|z-z_0|} \left| \frac{v_{\bar z \bar z}(z)}{\bar z - \bar z_0}\right|,  \\ \label{est83}
u_4(z,z_0) &= \frac{2}{|z-z_0|} \left| \frac{v_{\bar z}(z) - v_{\bar z}(z_0)}{(\bar z - \bar z_0)^2}\right|,  \\ \label{est84}
u_5(z,z_0) &= \frac{2}{|z-z_0|} \left| \frac{v(z) - v_{lin}(z,z_0)}{(\bar z - \bar z_0)^3}\right|.
\end{align}
This yields
\begin{equation} \label{est69}
|R^{2,2,2}_{z_0, \delta, \varepsilon}(\lambda) - R^{2,2,2}_{z_0, \delta, \delta/2}(\lambda)| \leq q_3 \log(\frac{\delta}{2 \varepsilon}) \| v \|_{C^2(\bar D)},
\end{equation}
where $z_0 \in D$, $0 < \varepsilon < \delta/2$. $\lambda \in \C \setminus \{ 0 \}$. Using \eqref{est66}, \eqref{est67}-\eqref{est69} with $\varepsilon = |\lambda|^{-1}$ we obtain \eqref{estm}. Lemma \ref{lem2} is proved.
\end{proof}

\begin{proof}[Proof of Lemma \ref{lem3}]
We write
\begin{align*}
W_{z_0} (\lambda) &= W^1_{z_0, \varepsilon} (\lambda) + W^2_{z_0, \varepsilon} (\lambda), \\
W^1_{z_0, \varepsilon} (\lambda) &= \int_{D \cap B_{z_0,\varepsilon}} e^{\lambda (z-z_0)^2 - \bar \lambda (\bar z - \bar z_0)^2} w(z) d \Ree z \, d \Imm z, \\
W^2_{z_0, \varepsilon} (\lambda) &= \int_{D \setminus B_{z_0,\varepsilon}} e^{\lambda (z-z_0)^2 - \bar \lambda (\bar z - \bar z_0)^2} w(z) d \Ree z \, d \Imm z,
\end{align*}
where $B_{z_0, \varepsilon} = \{ z \in \C : |z-z_0| < \varepsilon \}$. One sees that
\begin{align} \label{est73}
|W^1_{z_0, \varepsilon} (\lambda)| &\leq \int_{D \cap B_{z_0,\varepsilon}} \| w \|_{C(D)} d \Ree z \, d \Imm z = \pi \| w \|_{C(D)} \varepsilon^2, \\ \nonumber
W^2_{z_0, \varepsilon} (\lambda) &= \frac{-1}{2 \bar \lambda}\int_{D \setminus B_{z_0,\varepsilon}}  \frac{\partial}{\partial \bar z} \left(e^{\lambda (z-z_0)^2 - \bar \lambda (\bar z - \bar z_0)^2} \right) \frac{w(z)}{\bar z - \bar z_0} d \Ree z \, d \Imm z \\ \nonumber
&= W^{2,1}_{z_0, \varepsilon} (\lambda)+ W^{2,2}_{z_0, \varepsilon} (\lambda), \\ \nonumber
W^{2,1}_{z_0, \varepsilon} (\lambda) &= \frac{-1}{4 i \bar \lambda}\int_{\partial(D \setminus B_{z_0,\varepsilon})} e^{\lambda (z-z_0)^2 - \bar \lambda (\bar z - \bar z_0)^2} \frac{w(z)}{\bar z - \bar z_0} dz, \\ \nonumber
W^{2,2}_{z_0, \varepsilon} (\lambda) &= \frac{1}{2 \bar \lambda}\int_{D \setminus B_{z_0,\varepsilon}}  e^{\lambda (z-z_0)^2 - \bar \lambda (\bar z - \bar z_0)^2} \frac{\partial}{\partial \bar z} \left(\frac{w(z)}{\bar z - \bar z_0} \right) d \Ree z \, d \Imm z.
\end{align}
We have
\begin{align} \label{est71}
|W^{2,1}_{z_0, \varepsilon} (\lambda) | &\leq |\lambda|^{-1} a_1(D) \|w\|_{C(\bar D)} \log(3 \varepsilon^{-1}), 
\end{align}
\begin{subequations} \label{est72}
\begin{align}
|W^{2,2}_{z_0, \varepsilon} (\lambda) | &\leq |\lambda|^{-1} a_2(D) \|w\|_{C^1_{\bar z}(\bar D)} \log(3 \varepsilon^{-1}) \\
|W^{2,2}_{z_0, \varepsilon} (\lambda) | &\leq |\lambda|^{-1} a_2(D) \|w\|_{C(\bar D)} \log(3 \varepsilon^{-1}) \\ \notag
&\qquad+ |\lambda|^{-1} a_3(D,p)\|\frac{\partial w}{\partial \bar z}\|_{L^p(\bar D)},
\end{align}
\end{subequations}
for $z_0 \in D, \; \lambda \in \C \setminus \{ 0 \}, \; 0< \varepsilon \leq 1 , \; 2 < p <\infty$.

Using \eqref{est73}, \eqref{est71}, \eqref{est72} with $\varepsilon = |\lambda|^{-1}$ we obtain \eqref{estw}. This finishes the proof.
\end{proof}

\begin{proof}[Proof of Lemma \ref{lem4}]
Formula \eqref{estmuk} follows from the assumption on $\|g_{z_0,\lambda} v\|$ and from solving \eqref{eq14} by the method of successive approximations. \linebreak The proof of estimate \eqref{esth} follows from \eqref{estmuk} and Lemma \ref{lem3}. The proof follows.
\end{proof}

\section{An extension of Theorem \ref{maintheo}} \label{sec6}

As an extension of Theorem \ref{maintheo} for the case when we do not assume that $v_j|_{\partial D} \equiv 0, \; \frac{\partial}{\partial \nu} v_j|_{\partial D} \equiv 0, \; j= 1,2$, we give the following result.

\begin{prop} \label{mainprop}
Let $D \subset \R^2$ be an open bounded domain with $C^2$ boundary, let $v_1 , v_2 \in C^2(\bar D)$ with $\|v_j \|_{C^2(\bar D)} \le N$ for $j =1,2$, and $\Phi_1 , \Phi_2$ the corresponding Dirichlet-to-Neumann operators. Then, for any $0 < \alpha < \frac 1 5$,  there exists a constant $C = C(D, N, \alpha)$ such that the following inequality holds
\begin{equation} \label{estweak}
\|v_2 - v_1\|_{L^{\infty}(D)} \leq C \log(3 + \|\Phi_2 - \Phi_1\|_1^{-1} )^{-\alpha},
\end{equation}
where $\| A\|_1$ is the norm for an operator $A : L^{\infty}(\partial D) \to L^{\infty}(\partial D)$, with kernel $A(x,y)$, defined as $\|A\|_1 = \sup_{x,y \in \partial D} |A(x,y)| (\log(3 + |x-y|^{-1}))^{-1}$.
\end{prop}

All we need to know about $\| \cdot \|_1$ consists of the following:
\begin{itemize}
\item[\it i)] $\|A \|_{L^{\infty}(\partial D) \to L^{\infty}(\partial D)} \leq const(D) \|A\|_1$;
\item[\it ii)] by formula (4.9) of \cite{N1} one has
\begin{equation} \nonumber
\|v\|_{L^{\infty}(\partial D)} \leq const \|\Phi_{v} - \Phi_0\|_1.
\end{equation}
\end{itemize}

In order to prove Proposition \ref{mainprop} we need the following modified version of Lemma \ref{lem2}. We will call $(\partial D)_{\delta} = \{ z \in \C : dist(z, \partial D) < \delta \}$.

\begin{lem} \label{newlem2}
For $v \in C^2(\bar D)$ we have that
\begin{equation} \label{newestm}
|v(z_0) - \frac{2}{\pi}|\lambda| h^{(0)}_{z_0}(\lambda)| \leq \kappa_1(D)\delta^{-4} \frac{\log(3|\lambda|)}{|\lambda|}\|v\|_{C^2(\bar D)} + \kappa_2(D)\log(3+ \delta^{-1})\|v\|_{C( \partial D)},
\end{equation}
for $z_0 \in D \setminus (\partial D)_{\delta}$, $0 < \delta < 1$, $\lambda \in \C$, $|\lambda| \geq 1$.
\end{lem}

\begin{proof}[Proof of Lemma \ref{newlem2}]
Let $\chi_{\delta}$ be as in the proof of Lemma \ref{lem2}. We have in particular that
\begin{equation} \label{newestchi}
\|\chi_{\delta} \|_{C^k(\C)} \leq \delta^{-k} \|\chi\|_{C^k(\C)}, \; k \in \N.
\end{equation}
Let 
$$v_{lin} (z,z_0) = v(z_0) + v_z (z_0)(z-z_0) + v_{\bar z}(z_0) (\bar z - \bar z_0),$$
for $z, z_0 \in D$, $v_z = \frac{\partial v}{\partial z}$ and $v_{\bar z} = \frac{\partial v}{\partial \bar z}$.

We can write $h^{(0)}_{z_0}(\lambda) = S_{z_0, \delta} (\lambda) + R_{z_0,\delta} (\lambda)$, where
\begin{align*}
S_{z_0,\delta} (\lambda) &= \int_{\C} e_{\lambda,z_0}(z) v_{lin} (z,z_0) \chi_{\delta} (z-z_0) d\Ree z \, d \Imm z \\
&= \int_{\C} e^{i |\lambda|(z^2 + \bar z^2)} v_{lin} (e^{-i\varphi (\lambda)}z +z_0, z_0) \chi_{\delta} (z) d\Ree z \, d \Imm z, \\
R_{z_0, \delta}(\lambda) &= \int_{D} e_{\lambda,z_0}(z) \left(v(z) -v_{lin} (z,z_0) \chi_{\delta} (z-z_0)\right) d\Ree z \, d \Imm z
\end{align*}
where $\varphi(\lambda) = \frac 1 2 (\mathrm{arg}(\lambda) - \frac{\pi}{2})$, $e_{\lambda, z_0}(z) = e^{\lambda (z-z_0)^2 - \bar \lambda (\bar z- \bar z_0)^2}$, $z_0 \in D \setminus (\partial D)_{\delta}$, $\lambda \in \C$.

Using the stationary phase method and  the explicit construction of $\chi_{\delta}$ we obtain that
\begin{align} \label{newest65}
&v(z_0) = \frac{2}{\pi} \lim_{\lambda \to \infty} |\lambda| S_{z_0, \delta} (\lambda), \\ \label{newest66}
&|v(z_0) - \frac{2}{\pi} |\lambda| S_{z_0,\delta} (\lambda)| \leq \frac{\rho_1(D)}{\delta^4} \|v\|_{C^1(\bar D)} \| \chi\|_{C^4(\C)} |\lambda|^{-1},
\end{align}
$z_0 \in D \setminus (\partial D)_{\delta}$, $0 <\delta < 1$, $\lambda \in \C$, $|\lambda| \geq 1$.
Inequality \eqref{newest66} follows from

\begin{align*} 
|v(z_0) - \frac{2}{\pi} |\lambda| S_{z_0,\delta} (\lambda)| &\leq \frac{\rho_1(D) }{ |\lambda|}\|v_{lin}\|_{C^4(\bar D)} \|\chi_{\delta} \|_{C^4(\C)} \\
&\leq \frac{\rho_1(D) }{ |\lambda| \delta^4}\|v\|_{C^1(\bar D)} \|\chi \|_{C^4(\C)},
\end{align*}
where we used \cite[Lemma 7.7.3]{H} and \eqref{newestchi}.


Integrating by parts we can write
\begin{align*}
R_{z_0, \delta}(\lambda) &= -\frac{1}{2 \bar \lambda} \int_{D} \frac{\partial}{\partial \bar z}\left( e_{\lambda,z_0}(z) \right) \frac{\left(v(z) -v_{lin} (z,z_0) \chi_{\delta} (z-z_0)\right)}{\bar z - \bar z_0} d\Ree z \, d \Imm z \\
&= R^1_{z_0, \delta}(\lambda) + R^2_{z_0, \delta}(\lambda), \\
R^1_{z_0, \delta}(\lambda) &= \frac{-1}{4i \bar \lambda} \int_{\partial {D}} \! \! \! \! e_{\lambda,z_0}(z)  \frac{\left(v(z) -v_{lin} (z,z_0) \chi_{\delta} (z-z_0)\right)}{\bar z - \bar z_0} dz, \\
R^2_{z_0, \delta}(\lambda) &= \frac{1}{2 \bar \lambda} \int_{D} e_{\lambda,z_0}(z) \frac{\partial}{\partial \bar z} \left(\frac{\left(v(z) -v_{lin} (z,z_0) \chi_{\delta} (z-z_0)\right)}{\bar z - \bar z_0} \right) d\Ree z \, d \Imm z,
\end{align*}
for $z_0 \in D \setminus (\partial D)_{\delta}$, $\lambda \in \C \setminus \{ 0 \}$. In addition, we have that
\begin{align} \label{newest61}
\frac 2 \pi |\lambda| |R^1_{z_0,\delta} (\lambda)| &\leq \kappa_2(D) \log(3+ \delta^{-1}) \|v\|_{C(\partial D)}.
\end{align}

Formula \eqref{newest61} follows from the fact that $\chi_{\delta}(z-z_0) = 0$ for $z \in \partial D, \; z_0 \in D \setminus (\partial D)_{\delta}$ and from the estimate
\begin{align*}
\frac 2 \pi |R^1_{z_0, \delta}(\lambda)| &\leq \frac 2 \pi \frac{1}{|\lambda|} \int_{\partial D}\frac{|v(z)|}{|\bar z - \bar z_0|}|dz|  \leq \frac{\kappa_2(D) \log(3 + \delta^{-1})}{|\lambda|} \|v\|_{C(\partial D)}.
\end{align*}

We now write $R^2_{z_0, \delta}(\lambda) = \frac{1}{2 \bar \lambda} (R^{2,1}_{z_0, \delta, \varepsilon}(\lambda) + R^{2,2}_{z_0, \delta, \varepsilon}(\lambda))$, with
\begin{align} \label{newest63}
R^{2,1}_{z_0, \delta, \varepsilon}(\lambda) &= \int_{B_{z_0,\varepsilon}} e_{\lambda,z_0}(z) \frac{\partial}{\partial \bar z} \left(\frac{\left(v(z) -v_{lin} (z,z_0) \chi_{\delta} (z-z_0)\right)}{\bar z - \bar z_0} \right) d\Ree z \, d \Imm z \\ \label{newest64}
R^{2,2}_{z_0, \delta, \varepsilon}(\lambda) &= \int_{D_{z_0,\varepsilon}} \! \! \! \! \! \! \! e_{\lambda,z_0}(z) \frac{\partial}{\partial \bar z} \left(\frac{\left(v(z) -v_{lin} (z,z_0) \chi_{\delta} (z-z_0)\right)}{\bar z - \bar z_0} \right) d\Ree z \, d \Imm z,
\end{align}
where $B_{z_0, \varepsilon} = \{ z \in \C : |z-z_0| < \varepsilon \}$, $D_{z_0,\varepsilon} = D \setminus B_{z_0, \varepsilon}$. In \eqref{newest63}-\eqref{newest64} we assume that $z_0 \in D \setminus (\partial D)_{\delta}, \; 0 < \varepsilon < \delta, \; \lambda \in \C$.

The final part of the proof of estimate \eqref{newestm} consists in the following. We have, for $\varepsilon < \delta /2$,
\begin{align} \label{newest67}
|R^{2,1}_{z_0, \delta, \varepsilon}(\lambda)| &\leq \frac 7 2 \pi \|v\|_{C^2(\bar D)} \varepsilon^2,
\end{align}
exactly as in \eqref{est67},
\begin{align} \nonumber
R^{2,2}_{z_0, \delta, \varepsilon}(\lambda) &= -\frac{1}{2 \bar \lambda} \int_{D_{z_0,\varepsilon}} \frac{\partial}{\partial \bar z} \left( e_{\lambda,z_0}(z) \right) \frac{1}{\bar z - \bar z_0} \\ \nonumber
&\times \frac{\partial}{\partial \bar z} \left(\frac{\left(v(z) -v_{lin} (z,z_0) \chi_{\delta} (z-z_0)\right)}{\bar z - \bar z_0} \right) d\Ree z \, d \Imm z \\ \nonumber
&= -\frac{1}{2 \bar \lambda}(R^{2,2,1}_{z_0, \delta, \varepsilon}(\lambda)+ R^{2,2,2}_{z_0, \delta, \varepsilon}(\lambda)),
\end{align}
\begin{align} \nonumber
R^{2,2,1}_{z_0, \delta, \varepsilon}(\lambda) &= \frac{1}{2 i} \int_{\partial D_{z_0,\varepsilon}} \! \! \! \! \! \! \! \! \! e_{\lambda,z_0}(z) \frac{1}{\bar z - \bar z_0} \frac{\partial}{\partial \bar z} \left(\frac{\left(v(z) -v_{lin} (z,z_0) \chi_{\delta} (z-z_0)\right)}{\bar z - \bar z_0} \right) dz \\ \nonumber
&= -\frac{1}{2 i} \int_{\partial B_{z_0,\varepsilon}} e_{\lambda,z_0}(z) \frac{1}{\bar z - \bar z_0} \frac{\partial}{\partial \bar z} \left(\frac{v(z) -v_{lin} (z,z_0)}{\bar z - \bar z_0} \right) dz \\ \nonumber
&- \frac{1}{2 i} \int_{\partial D} e_{\lambda,z_0}(z) \frac{1}{\bar z - \bar z_0} \frac{\partial}{\partial \bar z} \left(\frac{v(z)}{\bar z - \bar z_0} \right) dz,
\end{align}

\begin{align*}
R^{2,2,2}_{z_0, \delta, \varepsilon}(\lambda) &= -\int_{D_{z_0,\varepsilon}}  e_{\lambda,z_0}(z) \\
&\times \frac{\partial}{\partial \bar z} \left( \frac{1}{\bar z - \bar z_0} \frac{\partial}{\partial \bar z} \left(\frac{\left(v(z) -v_{lin} (z,z_0) \chi_{\delta} (z-z_0)\right)}{\bar z - \bar z_0} \right)  \right) d\Ree z \, d \Imm z.
\end{align*}
We have, for $\varepsilon < \delta /2$
\begin{align} \label{newest68}
|R^{2,2,1}_{z_0, \delta, \varepsilon}(\lambda)| &\leq \frac{1}{2} \int_{\partial B_{z_0,\varepsilon}} \! \! \! \! \! \! \! \! \! \! \! \frac{|v(z)-v_{lin}(z,z_0)|}{|z-z_0|^3}|dz| + \frac{1}{2} \int_{\partial B_{z_0,\varepsilon}} \! \! \! \! \! \! \! \!  \! \frac{|v_{\bar z}(z)-v_{\bar z}(z_0)|}{|z-z_0|^2}|dz| \\ \nonumber
&+ \frac{1}{2} \int_{\partial D} \frac{|v(z)|}{|z-z_0|^3}|dz| + \frac{1}{2} \int_{\partial D} \frac{|v_{\bar z}(z)|}{|z-z_0|^2}|dz| \\ \nonumber
&\leq \frac 7 2 \pi \|v\|_{C^2(\bar D)}+ \frac{\rho_2(D)}{\delta^2} \|v\|_{C^1(\bar D)}, \\
|R^{2,2,2}_{z_0, \delta, \varepsilon}(\lambda)| &\leq |R^{2,2,2}_{z_0, \delta, \delta/2}(\lambda)| + |R^{2,2,2}_{z_0, \delta, \varepsilon}(\lambda) - R^{2,2,2}_{z_0, \delta, \delta/2}(\lambda)|, \\
|R^{2,2,2}_{z_0, \delta, \delta/2}(\lambda)| &\leq \frac{\rho_3(D)}{\delta^3} \|v\|_{C^2(\bar D)},\\ \nonumber
|R^{2,2,2}_{z_0, \delta, \varepsilon}(\lambda) - &R^{2,2,2}_{z_0, \delta, \delta/2}(\lambda)| \leq \sum_{j=1}^5 \int_{B_{z_0, \delta/2} \setminus B_{z_0, \varepsilon}} \! \! \! \! \! \! \! \! \! \! \! \! \! \! \! \! \! \! \! \! \! \! \! \! \! u_j (z,z_0) d \Ree z \, d \Imm z,
\end{align}
with $u_j$ defined as in \eqref{est80}-\eqref{est84}. This yields
\begin{equation} \label{newest69}
|R^{2,2,2}_{z_0, \delta, \varepsilon}(\lambda) - R^{2,2,2}_{z_0, \delta, \delta/2}(\lambda)| \leq \rho_4(D) \log(\frac{\delta}{2 \varepsilon}) \| v \|_{C^2(\bar D)},
\end{equation}
where $z_0 \in D \setminus (\partial D)_{\delta}$, $0 < \varepsilon < \delta/2$, $\lambda \in \C \setminus \{ 0 \}$. Using \eqref{newest66}, \eqref{newest61}, \eqref{newest67}-\eqref{newest69} with $\varepsilon = |\lambda|^{-1}$ we obtain \eqref{newestm} for $|\lambda| > \frac{2}{\delta}$.

Notice that only the estimation of $|\lambda| |R^2_{z_0, \delta}(\lambda)|$ requires $|\lambda| > \frac{2}{\delta}$. In that case one has
\begin{equation} \nonumber
\frac{2}{\pi}|\lambda| |R^2_{z_0, \delta}(\lambda)| \leq \rho_5(D) \delta^{-4} \frac{\log(3|\lambda|)}{|\lambda|}\|v\|_{C^2(\bar D)}.
\end{equation}

If $1 \leq |\lambda| \leq \frac{2}{\delta}$ we have that
\begin{align}
\frac{2}{\pi}|\lambda| |R^2_{z_0, \delta}(\lambda)| \leq \frac{\rho_6(D)N}{\delta}
\end{align}
and
\begin{align}
 \rho_5(D)\delta^{-4} \frac{\log(3|\lambda|)}{|\lambda|}\|v\|_{C^2(\bar D)} \geq  \frac{\rho_5(D)}{2 \delta^3}\log(6 \delta^{-1})\|v\|_{C^2(\bar D)},
\end{align}
where we used the fact that the function $\frac{\log(3s)}{s}$ is decreasing for $s > \frac{e}{3}$.

We now define
$$c' = 	\frac{2 \rho_6(D)N}{\rho_5(D) \log(6) \|v\|_{C^2(\bar D)}},$$
in order to have
$$\frac{2}{\pi}|\lambda| |R^2_{z_0, \delta}(\lambda)| \leq c' \rho_5(D) \delta^{-4} \frac{\log(3|\lambda|)}{|\lambda|}\|v\|_{C^2(\bar D)},$$
for $1 \leq |\lambda| \leq \frac{2}{\delta}$, $0 < \delta<1$.

Thus, taking $\kappa_1 = \max(\rho_5 ,c' \rho_5, \rho_1\|\chi\|_{C^4(\C)})$, we obtain estimation \eqref{newestm} for $|\lambda| \geq 1$ and $0 < \delta < 1$. This finish the proof of Lemma \ref{newlem2}.
\end{proof}

\begin{proof}[Proof of Proposition \ref{mainprop}]
Fix $0 < \alpha < \frac{1}{5}$, and $0 < \delta < 1$. We have the following chain of inequalities
\begin{align*}
\|v_2 - v_1\|_{L^{\infty}(D)} &= \max(\|v_2 - v_1\|_{L^{\infty}(D \cap (\partial D)_{\delta})}, \|v_2 - v_1\|_{L^{\infty}(D \setminus (\partial D)_{\delta})}) \\
&\leq C_1 \max \left( 2 N\delta + \|\Phi_2 - \Phi_1\|_1 , \frac{\log (3 \log(3+\| \Phi_2 -\Phi_1\|^{-1}))}{\delta^4 \log(3 + \|\Phi_2 - \Phi_1\|^{-1})} \right. \\ 
&\quad \left. +\log(3+ \frac 1 \delta) \|\Phi_2 - \Phi_1\|_1 + \frac{\log (3 \log(3+\| \Phi_2 -\Phi_1\|^{-1}))}{(\log(3 + \|\Phi_2 - \Phi_1\|^{-1}))^{\frac 1 2}}\right) \\
&\leq C_2 \max \left( 2 N\delta + \|\Phi_2 - \Phi_1\|_1 , \frac{1}{\delta^4} \log(3 + \|\Phi_2 - \Phi_1\|_1^{-1})^{-5\alpha } \right. \\
&\quad \left. +\log(3+ \frac 1 \delta) \|\Phi_2 - \Phi_1\|_1 +\frac{\log (3 \log(3+\| \Phi_2 -\Phi_1\|_1^{-1}))}{(\log(3 + \|\Phi_2 - \Phi_1\|_1^{-1}))^{\frac 1 2}} \right),
\end{align*}
where we followed the scheme of the proof of Theorem \ref{maintheo} with the following modifications: we make use of Lemma \ref{newlem2} instead of Lemma \ref{lem2} and we also use {i)-ii)}; note that $C_1 = C_1(D, N)$ and $C_2 = C_2(D, N, \alpha)$.

Putting $\delta = \log (3 + \| \Phi_2 - \Phi_1\|_1^{-1})^{-\alpha}$ we obtain the desired inequality
\begin{align} \label{estfinal}
&\|v_2 - v_1\|_{L^{\infty}(D)} \leq C_3 \log(3 + \|\Phi_2 - \Phi_1\|_1^{-1} )^{-\alpha},
\end{align}
with $C_3 = C_3(D, N, \alpha)$, $ \|\Phi_2 - \Phi_1\|_1 = \varepsilon \leq \varepsilon_1(D,N,\alpha)$ with $\varepsilon_1$ sufficiently small or, more precisely when $\delta_1 = \log (3 + \varepsilon_1^{-1})^{-\alpha}$ satisfies: 
$$\delta_1 < 1,  \qquad \varepsilon_1 \leq 2N \delta_1,  \qquad \log(3+\frac{1}{\delta_1}) \varepsilon_1 \leq \delta_1.$$

Estimate \eqref{estfinal} for general $\varepsilon$ (with modified $C_3$) follows from \eqref{estfinal} for $\varepsilon \leq \varepsilon_1(D,N,\alpha)$ and the assumption that $\|v_j\|_{L^{\infty}(\bar D)} \leq N$ for $j=1,2$. This completes the proof of Proposition \ref{mainprop}.
\end{proof}

\end{document}